\documentclass[a4paper, 11pt]{amsart}
\usepackage{amsmath}
\usepackage{amssymb}

\theoremstyle{plain}
\newtheorem{theorem}{\bf Theorem}[section]
\newtheorem{proposition}[theorem]{\bf Proposition}
\newtheorem{lemma}[theorem]{\bf Lemma}
\newtheorem{corollary}[theorem]{\bf Corollary}

\theoremstyle{definition}

\newtheorem{remark}[theorem]{\bf Remark}

\newcommand{\N}{\mathbf{N}}
\newcommand{\Z}{\mathbb Z}

\newcommand{\Q}{\mathbb Q}
\newcommand{\F}{\mathbb F}
\newcommand{\C}{\mathbb C}

\numberwithin{equation}{section}

\begin{document}

\title{On the Padovan sequence}
\author{Alain Faisant}



\begin{abstract}
The aim of this article is to give some properties of the so-called Padovan sequence $(T_n)_{n \ge 0}$ defined by
$$
T_{n+3} = T_{n+1} + T_n\quad\text{for all }n \in \N, \quad T_1=T_2=T_3=1
$$
that is divisibility properties, periods, identities.\\
\begin{center}-\end{center}
UJM Saint-Etienne, France, ICJ Lyon, France, CNRS-UMR 5208.\\
 E-mail :faisant@univ-st-etienne.fr
\end{abstract}

\maketitle
\bigskip
\section{Introduction}
\bigskip
Let $(T_n)_{n \ge 0}$ be the Padovan sequence, recursively defined by
$$T_{n+3} = T_{n+1} + T_n, \; n \ge 0$$\\
and the initial values  $T_0=0, T_1=1,T_2=1$. In this paper we consider a prime $p$, and give information on the index $\omega_p$ of first occurence of $p$  as divisor of any $T_n$ in the sequence. The principal result is that $\omega_p \le t_p \le p^{r_p}-1$, where $t_p$ is  the period modulo $p$ of  $(T_n)_{n \ge 0}$, and $r_p=1, 2, 3$ the degree of the extension $R_p / \F_p, R_p$ being the splitting field of the polynomial $T(X)=X^3-X-1 \mod p$ associated to the Padovan sequence. We give also some technical precisions (using in particular class field theory), examples, and identities relative to the sequence  $(T_n)_{n \ge 0}$. Many questions remain open.

\bigskip
\section{Definitions}
\bigskip
We consider the sequences recursively defined by $x_{n+3} = x_{n+1} + x_n, \; n \ge 0$.\\
The choice of $(x_0, x_1, x_2)$ determines the sequence : usually the Cordonnier-Padovan sequence is determined by the conditions $x_1=x_2=x_3=1$, corresponding to the choice $x_0=0$ denoted here by $(T_n)_{n\ge0}$ : the Padovan sequence. It is convenient for us to extend $T_{-1}=1, T_{-2}=0, T_{-3}=0$. It appears as sequence 
A000931 in Sloane's Online Encyclopedia of Integer Sequences \cite{Sloan}. \\
This sequence is also cited in \cite{Shannon} as $(R_n)_{n \ge 0} : R_n =T_{n-2}$.\\

The other choice $(3, 0, 2)$ gives the Perrin sequence $(P_n)_{n \ge 0}$ \cite{perrin}.\\

The first few terms of the sequence are\\
\begin{tabular}{c|cccccccccccccccccccccccc}
$n$&-3& -2&-1&0&1&2&3&4&5&6&7&8&9&10&11&12&13&14&15&16&17&18&19&20\\
\hline
$T_n $&0&0&1&0&1&1&1&2&2&3&4&5&7&9&12&16&21&28&37&49&65&86&114&151
\end{tabular}\\
The associated generating series is $$ \sum_{n\ge 0}T_nX^n=\frac {1-X}{1-X^2-X^3}$$
Let $K$ be a field, and consider 
$$
\mathcal{T}_K=\{(x_n)_{n\ge 0} : x_n \in K, x_{n+3}=x_{n+1}+x_n, \forall n \ge 0 \}
$$

This is a 3-dimensional $K$-vector space, with explicit evident isomorphism :
$$
\begin{array}{ccc}
K^3 & \stackrel{\varphi}{\longrightarrow} &\mathcal{T}_K\\
(x_0, x_1, x_2) & \longmapsto & (x_n)_{n \ge 0}
\end{array}
$$
The $K^3$-basis $\{(0, 1, 0), (1,0,1), (0,1,1) \}$ gives the $\mathcal{T}_K$-basis $\{(T_{n-2})_{n \ge 0}, (T_{n-1})_{n \ge 0}, (T_{n})_{n \ge 0}\}$.

\bigskip 

\section{Basis of $\mathcal{T}_K$}
We want to find a basis of $\mathcal{T}_K$ of geometrical progressions $(x^n)_{n\ge 0}$, i.e.
$$
x^{n+3}=x^{n+1}+x^n \;\;\;\forall n \ge 0
$$
this is equivalent to $x^3=x+1$ ;
 so we have to solve the equation $T(x)=0$ in $K$, where
$$ T(X)=X^3-X-1 \in K[X]$$
The general case is when $T$ admits 3 distinct roots $\alpha, \beta, \gamma \in K$: there exist 
$c_{\alpha}, c_{\beta}, c_{\gamma} \in K$ such that 

\begin{equation}\label{Tn}
T_n=c_{\alpha}\alpha^n + c_{\beta}\beta^n + c_{\gamma}\gamma^n, \quad n \ge 0
\end{equation}
corresponding to a $K$-linear system whose (Van der Monde) determinant is 
$$
\delta=\left |
\begin{matrix}
        &         1      & 1           &  1          &   \\
        & \alpha      & \beta     & \gamma   & \\
        & \alpha^2  &\beta^2 &  \gamma^2&     
 \end{matrix}
 \right |      
=(\alpha - \beta)(\beta-\gamma)(\gamma-\alpha) 
$$
and whose solution is

\begin{equation}
T_n= \frac{\alpha +1}{(\alpha-\beta)(\alpha-\gamma)}\alpha^n+\frac{\beta +1}{(\beta- \alpha)(\beta-\gamma)}\beta^n + \frac{\gamma +1}{(\gamma -\alpha)(\gamma-\beta)}\gamma^n
\end{equation}
Since $1+\alpha=\alpha^3$ we have the useful equivalent expression\\

\begin{equation}\label{delta}
-\delta T_n=(\beta-\gamma)\alpha^{n+3}+(\gamma-\alpha)\beta^{n+3}+(\alpha-\beta)\gamma^{n+3}
\end{equation}
\\
\begin{remark}\label{23}
If $T$ has a double root $\alpha=\beta$, a basis of $\mathcal{T}_K$ is 
$$
\{ (\alpha^n)_{n\ge 0}, (n\alpha^{n-1})_{n\ge 0}, (\gamma^n)_{n \ge 0} \}
$$
obtained by derivation of $\alpha^n$ with respect to the ``variable'' $n$ : \\
$n\alpha^{n-1}+(n-1)\alpha^{n-2}=n\alpha^{n+1}-\alpha^{n-2}=\alpha^{n+1}(n -\alpha^{-3})=\alpha^{n+1}(n- (1- \alpha^{-2}))$.\\
But $T'(\alpha)=3\alpha^2-1=0$ so $\alpha^{-2}=3$ : $n\alpha^{n-1}+(n-1)\alpha^{n-2}=\alpha^{n+1}(n+2)$ proving that $
 (n\alpha^{n-1})_{n\ge 0}$ lies in $\mathcal{T}_K$.
\end{remark}
\begin{remark}
For the Perrin sequence $(P_n)_{n \ge 0}$ mentioned above we have $P_n=\alpha^n+\beta^n+\gamma^n$ : if $p$ is prime we have $P_p=\alpha^p+\beta^p+\gamma^p=(\alpha +\beta +\gamma)^p=0$ in the splitting field of $T \mod p$ : so that
 $$n \mbox{\;\; prime} \Longrightarrow n|P_n$$
Perrin \cite{perrin} observed that, reciprocally,  for many, many, non prime $n$ we have $n \not | P_n$ : indeed the first counterexample is the ``pseudoprime'' $n= 271441$ : $n=521^2$.
\end{remark}

\bigskip 

\section{Calculus of the roots $\alpha, \beta, \gamma$ of $T(X)=X^3-X-1$}


\subsection{Rational case}\textrm{}\\
If $K=\Q$, then $T$ is irreducible and separable. Its discriminant is $Disc(T)= -23 <0$, so there exists one real root, and two complexes. Cardano's method gives explicitly the real root :
$$
\psi = \sqrt[3]{\frac{1}{2}+\frac{1}{6}\sqrt{\frac{23}{3}}}+\sqrt[3]{\frac{1}{2}-\frac{1}{6}\sqrt{\frac{23}{3}}}
$$
Here $\psi \simeq 1, 324718 ...$ is the {\bf plastic number}, a Pisot number (greater than 1, and conjugate less than 1 : here $\beta=\overline{\gamma}$ and $\beta \gamma=1/\psi <1$). Consequently by \ref{Tn} $\lim_{n \rightarrow + \infty}\frac{T_{n+1}}{T_n}= \psi$.


\subsection{Finite case}\textrm{}\\
Here $K=\F_p$ is the Galois field with $p$ elements, and we read modulo $p$ : 
$$T(X)=X^3-X-1 \in \F_p[X]$$

We need the following lemma, which expresses two roots rationaly from the third and the discriminant.
\begin{lemma}\label{abcd}\textrm{}\\
Let $\alpha, \beta, \gamma \in \overline{\F}_p$ be  the roots of $T(X)=X^3-X-1$\\
i) $\delta=(\alpha- \beta)(\beta - \gamma)(\gamma -\alpha)$ verifies \; $\delta^2=Disc(T)= -23 \mod p$,\\
ii)\\
$\bullet$ if $p\neq 2, 23$ :
\begin{equation}\label{rat}
\beta, \gamma= \frac{1}{2}[-\alpha \pm \frac{\alpha}{2\alpha + 3}\delta]
\end{equation}
$\bullet$ if $p=2 : \beta=\alpha^2, \gamma=\alpha^2+\alpha$,\\

\noindent
$\bullet$ if $p=23 : \alpha=3, \beta=\gamma=10$.
\end{lemma}

\begin{proof}{\textrm{}}\\
i) Classically $Disc(X^3+pX+q)=[(\alpha- \beta)(\beta - \gamma)(\gamma -\alpha)]^2= \delta^2= -4p^3-27q^2$.\\
ii) For $p=2, 23$ : direct calculation. Else one uses the relations
\begin{equation}
\begin{array}{lcc}
\alpha+ \beta +\gamma &= &0\\
\alpha \beta + \beta \gamma +\gamma \alpha&= &- 1\\
\alpha \beta \gamma& = &1\\
\alpha^3&=&\alpha + 1
\end{array}
\end{equation}
On the one hand $\beta + \gamma=-\alpha$, on the other hand $\delta\neq 0$, so $\beta - \gamma=\frac{\delta}{(\alpha - \beta)(\gamma - \alpha)}=\ldots = \frac{-\delta}{2\alpha +3}$ ; since $p\neq 2, \beta=\frac{1}{2}[(\beta + \gamma)+(\beta-\gamma)]$ and the formula follows. 
\end{proof}
\vspace{0.5cm}
Define $R_p=\F_p(\alpha, \beta,\gamma)$ as the splitting field of $T$ over $\F_p$, and $r_p=[R_p :\F_p]$ his degree.\\ There exist three cases :\\
$\bullet$ if $T$ has a root $\alpha \in \F_p$\\
\begin{enumerate}
\item either $-23$ is a square in $\F_p :\delta \in \F_p$, and by the lemma \ref{abcd} $\beta, \gamma \in \F_p$ : $R_p=\F_p$ and $r_p=1$ {\bf CASE 1}
(examples $p=23, 59, 101, 173, 211$)\\
\item or $-23$ is not a square : by the lemma \ref{abcd} $\beta, \gamma \in \F_p(\delta)=\F_{p^2}$ ; $R_p=\F_{p^2}$ and $r_p=2$\\
{\bf CASE 2} (examples $p=5,7, 11, 17)$\\
\end{enumerate}
\noindent
$\bullet$ if $T$ is irreducible in $\F_p[X]$\\
let $\alpha \in \overline{\F}_p$ be a root of $T$, then $R_p=\F_p(\alpha)=\F_{p^3}$ and $r_p=3$ : indeed $\F_p(\beta), \F_p(\gamma)$ are cubic extensions of $\F_p$, so are equal, to $\F_{p^3}$ ; moreover in this case $-23$ is a square in $\F_p$ : $-23=\delta^2$ is a square in $\F_{p^3}$, so in $\F_p$ :\\ 

\quad {\bf CASE 3}
(examples : $p=2, 3, 13, 29, 31$).\\

By quadratique reciprocity $-23$ is a square in $\F_p$ if and only if  $p$ is a square$\mod 23$, that is
$$ p\equiv 0, 1, 2, 3, 4, 6, 8, 9, 12, 13, 16, 18 \mod 23$$
In short
$$
\begin{array}{lclcc}
                                 &               & \alpha \in \F_p & R_p=\F_p     & r_p=1\\      
if -23 \mbox{\;\;square} \mod p   &\stackrel{\nearrow}{\searrow} &                              &           \\
                                     &            & T \mbox{\; irreducible}& R_p=\F_{p^3} & r_p=3\\
                                 &                &                                                             &          \\     
if -23 \mbox{\;\;not square} \mod p &             &\alpha \in \F_p& R_p=\F_{p^2}& r_p=2
\end{array} 
$$                                   

\bigskip 

\section{Prime divisors in the sequence $(T_n)_{n\ge 0}$}
We are interested in the first occurence of the prime $p$ in the Padovan sequence. In the case of the Fibonacci sequence $(F_n)_{n \ge 0}$, the answer is easy, linked with the order of $\Phi /\overline{\Phi} \mod p$ (where $\Phi$ is  the golden ratio, $\overline{\Phi}$ the conjugate of $\Phi$) : this is a divisor of $p-1$ or $p+1$ according to $p\equiv \pm1, p\equiv \pm2 \mod 5$. Here it is much more complicated.\\

We adopt the following notations
\begin{itemize}
\item $t_p$ = period of the sequence $(T_n)_{n\ge 0}$
\item $\Omega_p= \{ n \ge 1 : p \mbox{\;\;divide\;} T_n \}$
\item $\omega_p=\min \Omega_p$  : index of first occurence of $p$ in the sequence $(T_n)_{n \ge 1}$
\item $A_p= \Omega_p \cap \{1, 2, \ldots, t_p \}$
\item $a=o(\alpha), b=o(\beta), c=o(\gamma)$ (orders in $R_p^{\times}$ : divisors of $p^{r_p}-1$)
\end{itemize}


\subsection{Principal results}
\begin{theorem}\mbox{}\\
Let $R_p=\F_p(\alpha, \beta, \gamma), r_p=[R_p : \F_p], r_p=1, 2, 3$\\
1) for all $p$ : $t_p-3, t_p-2, t_p \in A_p$, hence $\omega_p \le t_p-3$\\
2) $\Omega_p=A_p+\N t_p$\\
3) if $p\neq 23 : t_p =LCM\{a, b, c \}$, and it divides $p^{r_p}-1$\\
4) if $p=23 : t_{23}=506=p(p-1)$, and $r_p=1$\\
5) moreover :\\
{\tiny $\clubsuit$}\;$r_p=1$ if and only if \;$p=x^2+23y^2$\\
{\tiny $\clubsuit$}\;if $r_p=2$ : \;$\bullet$ $\alpha \in \F_p$, and $t_p=b=c$\\
\mbox{\hspace{2.2cm}}$\bullet$ $b$ divide $(p+1)a$\\
\mbox{\hspace{2.2cm}}$\bullet$ $\omega_p \le (p+1)o(\alpha^3)-3$\\
{\tiny $\clubsuit$}\; if $r_p=3$ : $t_p=a=b=c$ divide $\frac{p^3-1}{p-1}=p^2+p+1$.
\end{theorem}

\begin{corollary}
For every prime $p$ there exists $n\ge 1$ such that $p$ divides $T_n$, and the index $\omega_p$ verifies : $\omega_p \le p^3-4$.
\end{corollary}
\begin{proof}
By definition $t_p=\min \{n\ge 1:T_n\equiv0, T_{n+1}\equiv 1, T_{n+2}\equiv 1 \mod p\}$. Along the proof we denote $t_p=t$.\\

1) Since $T_{t}=0, T_{t+1}=1, T_{t+2}=1$, by ``redshift'' we deduce $T_{t-1}=1, T_{t-2}=0, T_{t-3}=0$.\\

2) Let $n \in \Omega_p : n=ut+v, 0\le v \le t-1$ ; we have $T_t=T_0, T_{t+1}=T_1, T_{t+2}=T_2$, so induction gives $T_{t+m}=T_m$, then (induction on $u$) $T_{ut+m}=T_m$ ; so here $T_n=T_v$ ; $n \in \Omega_p \Rightarrow T_n=0$, so $T_v=0$ and $v \in \Omega_p \cap \{1, \ldots t \}=A_p$.\\

3) Let $s =LCM\{a, b, c\}$ ; $p\neq 23 \Rightarrow \delta\neq 0$. Apply \ref{delta} :
$ -\delta T_{s-3}=\beta-\gamma+\gamma-\alpha+\alpha-\beta=0$, so $T_{s-3}=0$.\\
Similarly $ -\delta T_{s-2}=(\beta-\gamma)\alpha+(\gamma-\alpha)\beta+(\alpha-\beta)\gamma=0$, so $T_{s-2}=0$, and $T_s=T_{s-2}+T_{s-3}=0$.\\
\ref{delta}: $-\delta T_{s-1}=(\beta-\gamma)\alpha^2+(\gamma-\alpha)\beta^2+(\alpha-\beta)\gamma^2=-\delta$, so $T_{s-1}=1$. Hence
$$ T_s=0, T_{s+1}=0, T_{s+2}=1 \Rightarrow t \le s $$
On the other hand, following \ref{Tn}\\
$$
\begin{array}{lclclclcc}
T_t &=&c_{\alpha}\alpha^t& +&c_{\beta}\beta^t&+&c_{\gamma}\gamma^t&=&0\\
T_{t+1}&=&c_{\alpha}\alpha^{t+1}& +&c_{\beta}\beta^{t+1}&+&c_{\gamma}\gamma^{t+1}&=&1\\
T_{t+2}&=&c_{\alpha}\alpha^{t+2}& +&c_{\beta}\beta^{t+2}&+&c_{\gamma}\gamma^{t+2}&=&1
\end{array}
$$
This proves that $(c_{\alpha}\alpha^t, c_{\beta}\beta^t, c_{\gamma}\gamma^t)$ is solution of the system
$$
\left\{
\begin{array}{rcrcrcc}
x&+&y&+&z&=&0\\
\alpha x&+&\beta y&+&\gamma z&=&1\\
\alpha^2x&+&\beta^2 y&+&\gamma^2z&=&1
\end{array}
\right.
$$
Its determinant is $(\alpha-\beta)(\beta-\gamma)(\gamma-\alpha)=\delta\neq 0$ and its unique solution is $(c_{\alpha}, c_{\beta}, c_{\gamma})$ ; so necessarly\\
\mbox{\hspace{2cm}}$c_{\alpha}\alpha^t=c_{\alpha}\\
\mbox{\hspace{2cm}}c_{\beta}\beta^t=c_{\beta}\\
\mbox{\hspace{2cm}}c_{\gamma}\gamma^t=c_{\gamma}$\\
$c_{\alpha}\neq 0$ since $\alpha \neq -1 : T(-1)=-1\neq 0$ ; we conclude that $\alpha^t=\beta^t=\gamma^t=1$, and $s=LCM\{o(\alpha), o(\beta), o(\gamma)\}$ divides $t$, and finally $s=t=t_p$.\\

4) The case $p=23$ is different (cf \ref{23}) : here $3T_n=4.3^n-4.10^n+8.n.10^{n-1}$ ; we have to solve the equation $(-2)^n=9n+1$ giving $n=503$ and $t_{23}=506=22\times23$.\\

5)\\
{\tiny $\clubsuit$}\;$r_p=1 \Leftrightarrow p=x^2+23y^2$ see below \ref{class} \\
{\tiny $\clubsuit$}\;$r_p=2$ : $\alpha \in \F_p, \beta,\gamma=\frac{1}{2}[-\alpha \pm \frac{\alpha}{2\alpha +3}\sqrt{-23}]$ are conjugate in $\F_p(\sqrt{-23})$ hence have same order $b=c$ ; moreover $\alpha=\frac{1}{\beta \gamma} \Rightarrow \alpha^b=1 : a |b$ and $t_p=b=c$.\\
The Frobenius $x\longmapsto x^p$ coincides with the conjugation : $\beta^p=\gamma,\; \beta^{p+1}=\beta \gamma=\frac{1}{\alpha}$, and $\beta^{(p+1)a}=1 : b|(p+1)a$.\\\\
Let $m\ge 1$ : $\alpha^{m(p+1)}=\alpha^{2m} ( \alpha^{p-1}=1) ; \beta^{m(p+1)}=\alpha^{-m}, \gamma^{m(p+1)}=\alpha^{-m}$ ; by \ref{delta} :\\ 
$-\delta T_{m(p+1)-3}=(\beta - \gamma)\alpha^{2m}+(\gamma-\alpha)\alpha^{-m}+(\alpha - \beta)\alpha^{-m}=(\beta-\gamma)\alpha^{-m}(\alpha^{3m}-1)$. The choice $m=o(\alpha^3)$ gives $m(p+1)-3 \in \Omega_p$. Example : p=7, $\alpha=5, \alpha^3=-1$, hence $2(7+1)-3=13 \in \Omega_p$.\\
{\tiny $\clubsuit$}\;$r_p=3$ : $\alpha,\beta, \gamma $ are conjugate so $a=b=c$ ; by Frobenius $x\longmapsto x^p : \alpha^p=\beta$ or $\gamma, \alpha^{p^2}=\gamma$ or $\beta$, and $\alpha \beta \gamma=1=\alpha \alpha^p\alpha^{p^2} : a=b=c |1+p+p^2$.
\end{proof}
\begin{remark}
The analog of the theorem may be done for the Tribonacci sequence : \\$\omega_p \le p^3-2  ;  r_p=1$ for $p=47, 53, 269,\ldots$ ; $r_p=2$ for $p=13, 17, 19, \ldots$ ; $r_p=3$ for $p=3, 5, 23 \ldots$
\end{remark}
\bigskip
It may happen that two consecutive terms are divisible by $p$.
\begin{proposition}
If $p\neq 23$, the following assertions are equivalent :\\
i) $\alpha^n=\beta^n=\gamma^n$\\
ii) $p|T_{n-3}$ and $p|T_{n-2}$.
\end{proposition}

\begin{proof}\mbox{}\\
$i) \Longrightarrow ii)$ : by \ref{delta}.\\
$ii) \Longrightarrow i)$ : by hypothesis and \ref{delta} :
\begin{equation}\label{a}
(\alpha-\beta)\gamma^n =(\gamma-\beta)\alpha^n+(\alpha-\gamma)\beta^n {\mbox \;\; and}
\end{equation}
\begin{equation}\label{b}
(\alpha-\beta)\gamma^{n+1} =(\gamma-\beta)\alpha^{n+1}+(\alpha-\gamma)\beta^{n+1}
\end{equation}

\noindent
\ref{a} $\Longrightarrow (\alpha-\beta)\gamma^{n+1} =(\gamma-\beta)\alpha^{n}\gamma+(\alpha-\gamma)\beta^{n}\gamma$=(by \ref{b})$ (\gamma-\beta)\alpha^{n+1}+(\alpha-\gamma)\beta^{n+1}$ ; so $(\gamma-\beta)\alpha^{n}(\alpha -\gamma)+(\alpha-\gamma)\beta^{n}(\beta - \gamma)=0$, and $(\gamma-\beta)(\alpha -\gamma)(\alpha^n -\beta^n)=0$ ; $p\neq 23 \Longrightarrow \alpha^n=\beta^n$. In the same way $\beta^n=\gamma^n$.
\end{proof}

\noindent
For example if $p=7$ we have $\alpha^{16}=\beta^{16}=\gamma^{16}=2$, so $7 |T_{13}$ and $7|T_{14}$.


\subsection{Examples}\mbox{}\\

\vspace{1cm}
\begin{tabular}{c|ll|c|c|c|c|c}
$p$ & roots and orders & &$t_p$ & $\omega_p$ & $A_p$ & $\sqrt{-23}$ & $r_p$\\
\hline
      &$\alpha$                    & $a=13$  &     &    &                    &               &  \\
  $3$  &$\beta=\alpha +1$  & $b=13$  & $13$& $6$ & $6, 10, 11, 13$ & $\pm 1$ & $3$ \\
      &$\gamma=\alpha -1$  & $c=13$  &     &    &                     &               &\\
\hline
      &$\alpha=5$                   & $a=6$    &    &    &                    &                &\\
 $7$   &$\beta=1+2\sqrt{-2}$   & $b=48$  & $48$ & $9$  &$9,13,14,16,25,29$&$\pm \sqrt{-2}$& $2$\\
      &$\gamma=1-2\sqrt{-2}$ & $c=48$ &      &    &$30,32,41,45,46,48$  &           & \\
\hline
       &$\alpha=6$                   & $a=10$  &      &    &                        &            &\\
  $11$ &$\beta=-3+2\sqrt{-1}$  & $b=120$ & $120$ & $25$ &$25,35,43,64,87,98$ &$\pm \sqrt{-1}$& $2$\\
      & $\gamma=-3-2\sqrt{-1}$& $c=120$  &       &      &$104,113,117,118,120$ &           & \\
\hline
       &$\alpha=4$                   & $a=29$     &      &      &                                &            & \\
$59$ &$\beta=13$                    & $b=58$    & $58$    & $42$ & $42,51,55,56,58$       & $\pm 6$ & 1\\
       &$\gamma=42$               & $c=58$    &       &     &                                &             &\\
\hline
\end{tabular}

\bigskip
The last three elements of $A_p$ are predicted by the theorem.

\subsection{The sequence $(T_n)_{n \ge 0}$ and class field theory}\label{class}\mbox{}\\
We refer to the text \cite{tim} which gives similar results for the Tribonacci sequence.\\
Here $K=\Q(\sqrt{-23})$ : the class number is $h_K=3$, the ring of integers is $\mathcal{O}_K=\Z[\sqrt{-23}]$. Denote by $R=\Q(X^3-X-1) \subset\C$ the splitting field of $T(X)=X^3-X-1$, and $\alpha \in \C$ a root of $T$.

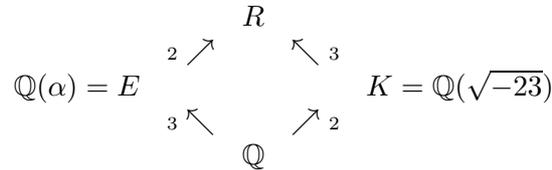
$$
\begin{array}{ccccc}
                     &                   &  R   &                  &                         \\
                     &{\mbox{\tiny 2}} \nearrow   &       &\nwarrow {\mbox{\tiny 3}} &                          \\
\Q(\alpha)=E &                   &       &                  &  K=\Q(\sqrt{-23})\\
                     &{\mbox{\tiny 3}} \nwarrow  &       &\nearrow {\mbox{\tiny 2}}  &                           \\
                     &                   & \Q  &                   &                           \\    
\end{array}      
$$

\noindent
{\bf{1}}- $\mathcal{O}_E=\Z[\alpha]$ : indeed $Disc(\alpha)=Disc(T)=-23$ is squarefree.\\
{\bf{2}}- Every $p\neq 23$ is unramified in $R/\Q$ : by Kummer--Dedekind theorem ($\mathcal{O}_E=\Z[\alpha]$) $p$ in unramified in $E$, also in the Galois closure $R$ of $E$.\\
{\bf{3}}- $P_{23}=\sqrt{-23} \mathcal{O}_K$ is unramified in $R/K$: we have $T \mod 23=(X-3)(X-10)^2$ ; By Kummer-Dedekind $23\mathcal{O}_E=P_1P_2^2$ with $3=e_1f_1+e_2f_2$, so $e=1=f_1=1, e_2=2, f_2=1$.\\
$R/K$ is Galois; we have $P_{23}\mathcal{O}_R=(\mathfrak{Q}_1\ldots \mathfrak{Q}_g)^e, efg=3$ ; then $e=1$ : else $e=3$ ; let $\mathfrak{P}|P_{23}$ in $R$ :\\
$e_{R/\Q}(\mathfrak{P})=e_{R/K}(\mathfrak{P}).e_{K/\Q}(P_{23})=3\times 2=6$, and\\
$e_{R/\Q}(\mathfrak{P}) =e_{R/E}(\mathfrak{P}).e_{E/\Q}(P_1 or P_2)$=[1 or 2]$\times $  [$ \leq 3]$\\

\noindent
necessarily $e_{R/E}(\mathfrak{P})=2$, and $e_{E/\Q}(P_i)= 3 (i=1,2)$, but $e_{E/\Q}(P_i)=1$ or 2 : contradiction.\\
{\bf{4}}- $K$ being imaginary, $R/K$ is unramified at infinity : so $R/K$ is abelian unramified and\\
$[R:K]=3=h_K$ ; we conclude that $R$ is the Hilbert class field of $K$.\\
{\bf{5}}- Application : \\
\begin{itemize}
\item  if $p\neq 23$ with $r_p=1$ we have $T \mod p=(X-\alpha)(X-\beta)(X-\gamma) \in \F_p[X]$ :
Kummer-Dedekind shows that $p\mathcal{O}_E=P_1P_2P_3$ : $p$ is totally decomposed 
in $E$,
and idem in the Galois closure $R$ of $E$,  a fortiori in $K$ : $p\mathcal{O}_K=\mathfrak{P}_1\mathfrak{P}_2$, and $\mathfrak{P}_1, \mathfrak{P}_2$ are totally decomposed in the Hilbert class field $R$ ; by the Artin reciprocity law $\mathfrak{P}_1$ is a principal ideal :
$\mathfrak{P}_1=(x+y\sqrt{-23})\mathcal{O}_K$ ; norm : $p=x^2+23y^2$.\\
\item Reciprocally if $p=x^2+23y^2$ we have $p\mathcal{O}_K=(x+y\sqrt{-23})\mathcal{O}_K.
(x-y\sqrt{-23}\mathcal{O}_K)$ : $p$ is totally decomposed in $K$.  By Artin reciprocity law , $(x+y\sqrt{-23})\mathcal{O}_K$ being principal, we conclude it is totally decomposed in $R$, and also in $E$ :
$p\mathcal{O}_E=P_1P_2P_3$ and Kummer-Dedekind show that $r_p=1$. 
\end{itemize}
We have proved
\begin{theorem}
$r_p=1 \Longleftrightarrow p=x^2+23y^2$
\end{theorem}
\noindent
Clearly $p=23$ is not a problem for the assertion!

\bigskip 

\section{Some identities for $(T_n)_{n\ge 0}$}
\subsection\\
For all $m$ $(T_{m+n})_{n \ge 0} \in \mathcal{T}_K$, so there exist $u,v,w \in K$ such that 
$T_{m+n}=uT_n+vT_{n-1}+wT_{n-2}$. Applying this for $n=0, n=1, n=2$, we obtain the system
$$
\left\{
\begin{array}{cccccl}
    &     & v  &    & =& T_m\\
 u & +  &  & w & =& T_{m+1}\\
 u& + & v &     & =& T_{m+2}
\end{array}
\right.
$$

Easy calculations give :
\begin{equation}
T_{m+n}=T_{m-1}.T_n +T_{m}.T_{n-1}+(T_{m+1}-T_{m-1}).T_{n-2}
\end{equation}\\

\noindent
As application let $m=n$ : 
$$T_{2n}=2T_{n-1}T_{n}+(T_{n+1}-T_{n-1})T_{n-2}$$
if $p |T_{n-2}$ and $p|T_{n-1}$ also $p|T_{2n}$. On can verify that $p=7|T_{13}$ and $p|T_{14}$, so $p|T_{30}$.

\subsection\\
The well-know formula related to the Fibonacci sequence $F_{n+1}F_{n-1}-F_n^2=(-1)^n$ (``Fibonacci-puzzle'') can be proved like this : one prove first that $\Phi^n=F_n\Phi+F_{n-1}$, and the analog with $\overline{\Phi}$ conjugate of $\Phi$ ; since $\Phi.\overline{\Phi}=-1$, one has $(\Phi.\overline{\Phi})^n=(-1)^n= (F_n\Phi+F_{n-1})(F_n\overline{\Phi}+F_{n-1})$ and the identity follows. We use here the same idea.\\
As $(\alpha^n)_{n\ge 0} \in \mathcal{T}_K$, there exist $u,v,w \in K$ such that 
$\alpha^n=uT_n+vT_{n-1}+wT_{n-2}$ : easy calculations give
\begin{equation}
\alpha^n=(\alpha^2-1)T_n+T_{n-1}+(1+\alpha-\alpha^2)T_{n-2}
\end{equation}
and the analog for $\beta^n, \gamma^n$. Write now $\alpha^n\beta^n\gamma^n=1$:\\

\bigskip
\bigskip
$[(\alpha^2-1)T_n+T_{n-1}+(1+\alpha-\alpha^2)T_{n-2}] \dots$\\

\hspace{1cm} $ \dots [(\beta^2-1)T_n+T_{n-1}+(1+\beta-\beta^2)T_{n-2}].
[(\gamma^2-1)T_n+T_{n-1}+(1+\gamma-\gamma^2)T_{n-2}]=1$
\\

The left member is a polynomial $P(T_n,T_{n-1},T_{n-2})=\sum_{(a,b,c)}c_{abc}T_n^aT_{n-1}^bT_{n-2}^c$ homogeneous of degree 3 : $a+b+c=3$, and the $c_{abc}$ are symmetric polynomials in $\alpha, \beta,\gamma$, so poynomials in the elementaries symmetric functions $s_1=\alpha +\beta +\gamma=0, s_2=\alpha \beta+\beta \gamma +\gamma \alpha=-1, s_3=\alpha \beta \gamma=1$. After long calculations we find

\begin{equation}
T_n^3+T_{n-1}^3+T_{n-2}^3-T_nT_{n-1}^2-T_n^2T_{n-2}+T_{n-1}^2T_{n-2}+2T_{n-1}T_{n-2}^2-3T_nT_{n-1}T_{n-2}=1
\end{equation}

\bigskip 

\section{Questions}We have many questions !\\
$Q_1-$ In the case $r_p=1$ it seems that $b=c=2a$ : is it true ?\\
$Q_2-$ In the case $r_p=2$ it seems that $b=c=(p+1)a$ : is it true ?\\
$Q_3-$ When $r_p=3$ : $a=b=c$ ; some examples show that $a=1+p+p^2$ : is it true ?\\
$Q_4-$ What is the exact arithmetical nature of $\omega_p$, of $A_p$ ?\\
$Q_5-$ Primitive divisors : the computations give many exceptions :
$$\; n=5, 7, 10, 11, 12, 13, 14, 16, 21, 23, 32, 33, 45, ...  \;$$ 
Are there finitely many ?\\
$Q_6-$ $T_n$ is prime for $$n=4, 5, 6, 8, 9, 15, 20, 31, 38, ...$$ \\
Are there infinitely many ?\\
$Q_7-$ We find the following square in $(T_n)_{n \ge 0}$ : $$T_{10}=9, T_{12}=16, T_{16}=49$$
Are there finitely many squares ? \\
Observe that $\sqrt{9}=3=T_6, \sqrt{16}=4=T_7, \sqrt{49}=7=T_9$ : is it a coincidence ?\\ 

Finally the author thanks F. Nuccio for his (im)pertinent galoisian observations !
\begin{center}
{\tiny $\clubsuit$}
\end{center}


\providecommand{\bysame}{\leavevmode\hbox to3em{\hrulefill}\thinspace}
\providecommand{\MR}{\relax\ifhmode\unskip\space\fi MR }
\providecommand{\MRhref}[2]{%
  \href{http://www.ams.org/mathscinet-getitem?mr=#1}{#2}
}
\providecommand{\href}[2]{#2}

\vspace{1cm}
\end{document}